\documentclass[11pt]{amsart}
\usepackage{amsfonts}
\usepackage{amsmath}
\usepackage{amsthm}
\usepackage{bm}
\usepackage{graphicx}
\usepackage{natbib}
\usepackage{bigints}
\usepackage{color}
\usepackage{comment}
\usepackage{enumitem}
\usepackage{mathrsfs}

\numberwithin{equation}{section}

\allowdisplaybreaks

\theoremstyle{plain}
\newtheorem{theorem}{Theorem}[section]
\newtheorem{lemma}[theorem]{Lemma}
\newtheorem{proposition}[theorem]{Proposition}
\newtheorem{corollary}[theorem]{Corollary}
\theoremstyle{definition}
\newtheorem{definition}[theorem]{Definition}
\newtheorem{remark}[theorem]{Remark}
\newtheorem{example}[theorem]{Example}
\newtheorem{assumption}[theorem]{Assumption}

\def\beqn{\begin{equation}}
\def\beqn*{$$}
\def\eeqn{\end{equation}}
\def\eeqn*{$$}

\newcommand{\BX}{{\bf X}}

\newcommand{\BZ}{{\bf Z}}
\newcommand{\BV}{{\bf V}}
\newcommand{\bx}{{\bf x}}
\newcommand{\by}{{\bf y}}
\newcommand {\ba}{{\bf a}}
\newcommand {\bb}{{\bf b}}
\newcommand {\bp}{{\bf p}}

\newcommand{\reals}{{\mathbb R}}
\newcommand{\bbr}{\reals}

\newcommand{\vep}{\varepsilon}

\newcommand{\calR}{\mathcal{R}}

\newcommand{\scrS}{\mathscr{S}}
\newcommand{\scrL}{\mathscr{L}}
\newcommand{\scrR}{\mathscr{R}}
\newcommand{\scrG}{\mathscr{G}}
\newcommand{\scrD}{\mathscr{D}}

\newcommand{\eid}{\stackrel{d}{=}}

\newcommand{\vconv}{\stackrel{v}{\rightarrow}}

\newcommand{\one}{{\bf 1}}

\begin{document}

\title[Multivariate Subexponential Distributions]
{Multivariate Subexponential Distributions and Their Applications}

\author{Gennady Samorodnitsky}
\address{School of Operations Research and Information Engineering\\
and Department of Statistical Science \\
Cornell University \\
Ithaca, NY 14853}
\email{gs18@cornell.edu}

\author{Julian Sun}
\address{School of Operations Research and Information Engineering\\
Cornell University \\
Ithaca, NY 14853}
\email{ys598@cornell.edu}

\thanks{This research was partially supported by the ARO
grant  W911NF-12-10385 at Cornell University.}

\keywords{heavy tails, subexponential distribution, regular variation,
  multivariate, insurance portfolio, ruin probability
\vspace{.5ex}}

\begin{abstract}
 We propose a new definition of a multivariate subexponential
 distribution. We compare this definition with the two existing
 notions of multivariate subexponentiality, and compute the
 asymptotic  behaviour of the ruin probability in the context of an
 insurance portfolio, when multivariate subexponentiality
 holds. Previously such results were available only in the case of
 multivariate regularly varying claims.
\end{abstract}

\maketitle

\section{Introduction}
\label{sec:intro}

Subexponential distributions are commonly viewed as the most general
class of heavy tailed distributions. The notion of subexponentiality
was introduced by \cite{chistyakov:1964} for distributions supported
by $[0,\infty)$; if $F$ is such a distribution, and $X_1$, $X_2$ are
i.i.d. random variables with the law $F$, then $F$ is
subexponential if
\begin{equation} \label{e:def.1dim}
\lim_{x\to\infty} \frac{P(X_1+X_2>x)}{P(X_1>x)}=2\,.
\end{equation}
The notion of subexponentiality was later extended to distributions
supported by the entire real line (and not only by the positive
half-line); see e.g. \cite{willekens:1986}.
The best known subclass of subexponential distributions is that of
regularly varying distributions, but the membership in the class of
subexponential distributions does not require power-like tails; we
review the basic information on one-dimensional subexponential
distributions in Section \ref{sec:1dim}.

The definition \eqref{e:def.1dim} of subexponential distributions
means that the sum of two i.i.d. random variables with a subexponential
distribution is large only when one of these random variables is
large. The same turns out to be true for the sum of an arbitrary
finite number of terms and, in many cases, for the sum of a random
number of terms. Theoretically, this leads to the ``single large
jump'' structure of large deviations for random walks with
subexponentially distributed steps; see
e. g. \cite{foss:konstantopoulos:zachary}.  In practice, this has
turned out to be particularly important in applications to ruin
probabilities. In ruin theory the situation where the claim sizes
(often assumed to be independent with identical distribution) have
a subexponential distribution is usually referred to as the
non-Cram\'er case. The ``single large jump'' property of
subexponential distributions leads to a
well known form of the asymptotic behaviour of the ruin probability,
and to a particular structure of the surplus path leading to the ruin;
see e.g. \cite{embrechts:kluppelberg:mikosch:1997} and
\cite{asmussen:2000}.

It is desirable to have a notion of a multivariate subexponential
distribution. The task is of a clear theoretical interest, and it is
of an obvious interest in applications. A typical insurance company,
for instance, has multiple insurance portfolios, with dependent
claims, so it would be useful if one could build a model in which
claims could be said to have a multivariate subexponential
distribution. Recall that there exists a well developed notion of a
multivariate distribution with regularly varying tails; see
e.g. \cite{resnick:2007}. In comparison, a notion of a multivariate
subexponential distribution has not been developed to nearly the same
extent. To the best of our knowledge, a notion of multivariate
subexponentiality has been introduced twice, in
\cite{cline:resnick:1992} and in \cite{omey:2006}. Both of these
papers define a class (or classes) of multivariate distributions that
extend the the one-dimensional notion of a subexponential distribution
in a natural way. They show that their notions of multivariate
subexponentiality possess multidimensional analogs of important
properties of one-dimensional subexponential
distributions. Nonetheless, these notions have not become as widely
used as that of, say, a multivariate distribution with regularly
varying tails. In
this paper we introduce yet another notion of multivariate
subexponential distribution. As the reader will observe, this notion
is created with ruin probability applications in mind. We hope,
therefore, that this notion will turn out to be useful in that
area. However, we also hope that the notion we introduce will be found
useful in other areas as well.

This paper is organized as follows. In Section \ref{sec:1dim} we
review the basic properties of one-dimensional subexponential
distributions, in order to have a benchmark for the properties we
would like a multivariate subexponential
distribution to have. In Section \ref{sec:previous} we discuss the
definitions of multivariate subexponentiality of
\cite{cline:resnick:1992} and in \cite{omey:2006}. Our notion of
multivariate subexponential distributions is introduced in Section
\ref{sec:our.def}. Some applications of that notion to multivariate
ruin problems are discussed in Section \ref{sec:ruin}.

\section{A review of one-dimensional subexponentiality}
\label{sec:1dim}

In this section we review the basic properties of one-dimensional
subexponential distributions. We denote the class of such
distributions (and random variables with such distributions) by
$\scrS$. Unless stated explicitly, we do not assume
anymore that a random variable with a subexponential distribution $F$
is nonnegative; such a random variable (or its distribution) is called
subexponential if the nonnegative random variable $X_+=\max(X,0)$ is
subexponential. Most of the not otherwise attributed facts stated
below can be found in
\cite{embrechts:goldie:veraverbeke:1979}. We use the standard notation
$\bar F=1-F$ for the tail of a distribution $F$.

If a distribution $F\in \scrS$, then $F$ is {\it long-tailed}: for any
$y\in\bbr$,
\begin{equation} \label{e:long.tail}
\lim_{x\to\infty} \frac{\bar F(x+y)}{\bar F(x)} =1
\end{equation}
(implicitly assuming that $\bar F(x)>0$ for all $x$.) The class of
all long-tailed distributions is denoted by $\scrL$. Note that $\scrS$
is a proper subset of $\scrL$; see
e.g. \cite{embrechts:goldie:1980}. Furthermore, the class $\scrL$ of
long-tailed distributions is
closed under convolutions, while the class $\scrS$ of subexponential
distributions is not, see \cite{leslie:1989}.

A distribution $F$ has a regularly varying right tail if there is
$\alpha\geq 0$ such that for every $b>0$
\begin{equation} \label{e:regvar}
\lim_{x\to\infty} \frac{\bar F(bx)}{\bar F(x)} = b^{-\alpha}\,,
\end{equation}
and the parameter $\alpha$ is the exponent of regular variation. The
class of distributions with a regularly varying right tail is denoted
by $\scrR$ (or $\scrR(\alpha)$ if we wish to emphasize the exponent of
regular variation.) Then $\scrR\subset\scrS$. If one views
$\scrR$ as the class of distributions with ``power-like'' right tails,
all distributions with ``power-like'' right tails are
subexponential. This statement, however, should be treated
carefully; other classes of distributions can be referred to as having
"power-like" right tails, and not all of them form subclasses of
$\scrS$.  Indeed, consider the class $\scrD$ of distributions with
{\it    dominated varying tails}, defined by the property
\begin{equation} \label{e:domvar}
\liminf_{x\to\infty} \frac{\bar F(2x)}{\bar F(x)} >0\,.
\end{equation}
One could view a distribution $F\in \scrD$ as having a
``power-like'' right tail. However, $\scrD\not\subset\scrS$. We note,
on the other hand, that it is still true that
$\scrD\cap\scrL\subset\scrS$; see \cite{goldie:1978}.

Many distributions that do not have
``power-like'' right tails are subexponential as well. Examples
include the log-normal distribution, as well as the Weibull
distribution with the shape parameter smaller than 1; see
e.g. \cite{pitman:1980}.

Let $X_1,X_2,\ldots$ be i.i.d. random variables with a subexponential
distribution. The defining property \eqref{e:def.1dim} extends, automatically, to
any finite number of terms, i.e.
\begin{equation} \label{e:nterms}
\lim_{x\to\infty} \frac{P(X_1+\ldots +X_n>x)}{P(X_1>x)}=n\ \
\text{for any $n\geq 1$.}
\end{equation}
Moreover, the number of terms can also be random. Let $N$ be a random
variable independent of the i.i.d. sequence $X_1,X_2,\ldots$ and
taking values in the set of nonnegative integers. If
\begin{equation} \label{e:exp.mom}
E\tau^N<\infty \ \ \text{for some $\tau>1$,}
\end{equation}
then
\begin{equation} \label{e:rand.terms}
\lim_{x\to\infty} \frac{P(X_1+\ldots +X_N>x)}{P(X_1>x)}=EN\,.
\end{equation}

The classical one-dimensional (Cram\'er-Lundberg)  ruin problem can be
described as follows. Suppose that an insurance company has an initial
capital $u>0$. The company receives a stream of premium
income at a constant rate $c>0$ per unit of time. The company has to
pay claims that arrive according to a rate $\lambda$ Poisson
process. The claim sizes are assumed to be i.i.d. with a finite mean
$\mu$ and independent of the arrival process. If $U(t)$ is the capital
of the company at time $t\geq 0$, then the ruin probability is defined
as the probability the company runs out of money at some point. This
probability is, clearly, a function of the initial capital $u$, and it
is often denoted by
$$
\psi(u) = P\bigl( U(t)<0 \ \ \text{for some $t\geq 0$}\bigr)\,.
$$
The {\it positive safety loading}, or the  {\it net profit
  condition},
$$
\rho := \frac{c}{\lambda\mu}-1>0
$$
says that, on average, the company receives more in premium income
than it spends in claim payments. If the net profit condition fails,
then an eventual ruin is certain. If the net profit condition holds,
then the ruin probability is a number in $(0,1)$, and its behaviour
for large values of the initial capital $u$ strongly depends on the
properties of the distribution $F$ of the claim sizes.  Let
$$
F_I(x) =\frac{1}{\mu} \int_0^x \bar F(y)\, dy, \ x\geq 0
$$
be the integrated tail distribution. If $F_I\in \scrS$, then
\begin{equation} \label{e:subexp.ruin}
\psi(u) \sim \rho^{-1} F_I(u) \ \ \text{as $u\to\infty$;}
\end{equation}
see Theorem 1.3.6 in \cite{embrechts:kluppelberg:mikosch:1997}.

\section{Existing definitions of multivariate subexponentiality}
\label{sec:previous}

The first known to us definition of multivariate subexponential distributions was
introduced by \cite{cline:resnick:1992}. They consider distributions
supported by the entire $d$-dimensional space $\bbr^d$ (and not only
by the nonnegative orthant). That paper defines both multivariate
subexponential distributions, and multivariate
{\it exponential distributions}. In our discussion here we only
consider the subexponential case. The definition is tied to a function
$\bb (t) = (b_1(t), \dots, b_d(t))$ such that $b_i(t) \to \infty$ as
$t \to \infty$ for $i = 1, \ldots, d$.

One starts with defining the class of long-tailed distributions,
i.e. a multivariate analog of the class $\scrL$ in
\eqref{e:long.tail}. Let $E = [ - \infty, \infty ]^d \setminus \left\{
  - \pmb{\infty} \right\}$, and let $\nu$ be a finite measure on $E$
concentrated on the purely infinite points, i.e. on $\{ - \infty, \infty \}^d
 \setminus \left\{   - \pmb{\infty} \right\}$, and such that
$\nu(\bx\in E:\, x_i=\infty)>0$ for each $i=1,\ldots, d$.
Then a probability distribution $F$ is
said to belong to the class $\scrL(\nu;\bb)$ if, as $t\to\infty$,
\begin{equation} \label{e:mult.longtail}
tF\bigl( \bb(t)+\cdot\bigr) \vconv \nu
\end{equation}
vaguely in $E$ (see \cite{resnick:1987} for a thorough treatment of
vague convergence of measures.) The class of subexponential
distributions (with respect to the same function $\bb$ and the same
measure $\nu$) is defined to be that subset $\scrS(\nu;\bb)$
of distributions $F$ in $\scrL(\nu;\bb)$ for which
$$
tF\ast F\bigl( \bb(t)+\cdot\bigr) \vconv 2\nu
$$
vaguely in $E$.

Corollary 2.4 in \cite{cline:resnick:1992} shows that $F\in
\scrS(\nu;\bb)$ if and only if $F\in \scrL(\nu;\bb)$ and the marginal
distribution $F_i$ of $F$ is in the one-dimensional subexponential
class $\scrS$ for each $i=1,\ldots, d$.

It is shown in \cite{cline:resnick:1992} that the distributions
in $\mathscr{S} (\nu, \mathbf{b})$ possess the natural multivariate
extensions of the properties of the one-dimensional subexponential
distributions mentioned in Section \ref{sec:1dim}. For example, if
$F\in \scrS(\nu,\bb)$, then for any $n\geq 1$, $F^{\ast n}\in
\scrS(n\nu,\bb)$. More generally, if $N$ is a random variable
satisfying \eqref{e:exp.mom}, and $H=\sum_{n=0}^\infty P(N=n) F^{\ast n}$,
then $H\in \scrS(EN\nu,\bb)$.

The distributions in $\scrS(\nu,\bb)$ also possess the right relation
with the  distributions with multivariate regularly varying tails. It
is natural, in this situation, to consider only distributions
supported by the nonnegative quadrant $R_+^d = [0,\infty)^d$. Recall
that any distribution $F$ supported by $R_+^d$ is said to have
regularly varying tails if
there is a Radon measure $\mu$ on $[0,\infty]^d \setminus \{ {\bf
  0}\}$ concentrated on finite points, and a function $\bb$ as above
such that, as $t\to\infty$,
\begin{equation} \label{e:regvar.d}
tF\bigl( \bb(t)\cdot\bigr)\vconv \mu
\end{equation}
vaguely in $[0,\infty]^d \setminus \{ {\bf   0}\}$; see
\cite{resnick:2007}.  Note that \eqref{e:regvar.d} allows for
different scaling in different directions, hence also different
marginal exponents of regular variation. This situation is sometimes
referred to as {\it non-standard regular variation}.
If we denote by $\scrR(\mu,\bb)$ the class of  distributions with
regularly varying tails satisfying \eqref{e:regvar.d}, then, as
shown in \cite{cline:resnick:1992}, $\scrR(\mu,\bb)\subset
\scrS(\nu,\bb)$ for some $\nu$.

As mentioned above, this definition of multivariate subexponentiality
requires, beyond marginal subexponentiality for all components, only
the joint long tail property \eqref{e:mult.longtail}. This property,
together with the nature of the limiting measure, makes this notion
somewhat inconvenient in applications, because it is not easy to see
how to use it on sets  in $\mathbb{R}^d$ that are not ``asymptotically
rectangular''.

Another observation worth making is that in
probability theory, many well established multivariate extensions of important
one-dimensional notions have a ``stability property'' with respect to
projections on one-dimensional subspaces (i.e., with respect to taking
linear combinations of the components.) Specifically, if the
distribution of a random vector $(X^{(1)},\ldots, X^{(d)})$ has, say,
a property $\scrG_d$ (the subscript $d$ specifying the dimension in
which the property holds), then the distribution of any
(non-degenerate) linear combination $\sum_1^d a_iX^{(i)}$ has the
property $\scrG_1$. This is true, for instance, for multivariate
regular variation, multivariate Gaussianity, stability and infinite
divisibility. Unfortunately, the definition of multivariate
subexponentiality by $\scrS(\nu,\bb)$ does not have this feature, as
the following example shows.

\begin{example} \label{crnonlin}
Consider a $2$-dimensional random vector $(X, Y)$ with nonnegative
coordinates such that $P(X + Y
= 2^n) = 2^{-(n+1)}$ for $n \geq 0$, with the mass distributed
uniformly on the simplex $\left\{ (x,y):\,  x , y \geq 0,\,  x + y = 2^n
\right\}$ for each $n \geq 0$. It is elementary to check that $X, Y
\in \mathscr{L} \cap \mathscr{D} \subset \mathscr{S}$. Furthermore,
for $2^n\leq x\leq 2^{n+1}$, $n=0,1,2,\ldots$ we have
$$
P(X>x)=P(Y>x) = 2^{-(n+1)}-\frac{x}{3}2^{-(2n+1)}
= 2P(X>x,Y>x)\,.
$$
If we define a function $b$ by $tP(X>b(t))=1$ for $t\geq 2$, then it
immediately follows that $(X, Y) \in \mathscr{L}(\nu; \mathbf{b})$
with $\bb(t)=(b(t),b(t))$ and
$$
\nu=\frac12\delta_{(-\infty,\infty)}+\frac12\delta_{(\infty,-\infty)}
+\frac12\delta_{(\infty,\infty)}\,,
$$
and the result of \cite{cline:resnick:1992} tells us that
$(X, Y) \in \mathscr{S}(\nu; \mathbf{b})$. It is clear, however, that
$$
\liminf_{x\to\infty} \frac{P(X+Y>x+1)}{P(X+Y>x)}=\frac12\,,
$$
so $X + Y$ does not even have a long-tailed, let alone
subexponential, distribution.
\end{example}

The second existing definition of multivariate subexponentiality we
are aware of is due to \cite{omey:2006}. Once again, this definition
concentrates on rectangular regions. The paper presents 3 versions of
the definition. The versions are similar, and we concentrate only on
one of them. Let $F$ be a probability distribution supported by the
positive quadrant in $\bbr^d$. Then one says that   $F \in
S(\mathbb{R}^d)$ if for all $\mathbf{x} \in (0,\infty]^d$ with $\min(x_i)
< \infty$,
\begin{align} \label{e:omey.S}
\lim_{t \to \infty} \frac{\overline{F^{*2}}(t \mathbf{x})}{\overline{F}(t \mathbf{x})} = 2.
\end{align}

This definition, like the definition of
\cite{cline:resnick:1992}, has the following property: a distribution
$F \in S(\mathbb{R}^d)$ if and only if each marginal distribution $F_i$ of
$F$ is a one-dimensional subexponential distribution, and a
multivariate long-tail property holds. In the present case the
long-tail property is
\begin{align} \label{omeyL}
\lim_{t \to \infty} \frac{\overline{F}(t \mathbf{x} -
  \mathbf{a})}{\overline{F}(t \mathbf{x})} = 1
\end{align}
for each $\mathbf{x} \in (0,\infty]^d$ with $\min(x_i)
< \infty$ and each $\ba\in [0,\infty)^d$. This follows from  Theorem 7
and Corollary 11 in \cite{omey:2006}.

The following statement shows that, in fact, the definition
\eqref{omeyL} of multivariate subexponentiality requires {\it only}
marginal subexponentiality of each coordinate.

\begin{proposition} \label{nonomey}
Let $F$ be a probability distribution supported by the
positive quadrant in $\bbr^d$.  Then $F \in
S(\mathbb{R}^d)$ if and only if all marginal distributions $F_i$ of
$F$ are subexponential in one dimension.
\end{proposition}
\begin{proof}
By choosing $\bx$ with only one finite coordinate, we immediately see
that if $F \in S(\mathbb{R}^d)$, then $F_i \in \mathscr{S}$ for each $i
= 1, \dots, d$.

In the other direction, we know by the results of \cite{omey:2006},
that only the long-tail property \eqref{omeyL} is needed, in addition
to the marginal subexponentiality, to establish that $F \in
S(\mathbb{R}^d)$. Therefore, it is enough to check that  the long-tail
property \eqref{omeyL} follows from the marginal subexponentiality. In
fact, we will show that, if each $F_i$ is long-tailed, i.e. satisfies
\eqref{e:long.tail}, $i=1,\ldots, d$, then \eqref{omeyL} holds as
well.

Let $\epsilon > 0$. Fix $\mathbf{x} = (x_1, \ldots, x_d) \in
(0,\infty)^d$ (allowing some of the components of $\bx$ be infinite
only leads to a reduction in the dimension), and
 and $\mathbf{a} = (a_1, \ldots, a_d)  \in [0,\infty)^d$.

Since $F_i \in \mathscr{L}$, $i=1,\ldots, d$, for sufficiently large
$t$ we have
\begin{align*}
0 \leq \overline{F_i}(tx_i - a_i) - \overline{F_i}(tx_i) < \epsilon \overline{F_i}(tx_i)
\end{align*}
for $i = 1, \ldots, d$. Further, it is clear that
\begin{align*}
0 \leq \overline{F}(t \mathbf{x} - \mathbf{a}) - \overline{F}(t
  \mathbf{x}) \leq
\sum_{i=1}^d \left( \overline{F_i}(tx_i - a_i) - \overline{F_i}(tx_i)
  \right).
\end{align*}

Hence for sufficiently large $t$,
\begin{align*}
0 \leq \frac{\overline{F}(t \mathbf{x} - \mathbf{a}) - \overline{F}(t
  \mathbf{x})}{\overline{F}(t \mathbf{x})}
& \leq \sum_{i=1}^d\frac{\left(\overline{F_i}(tx_i -a_i)
-\overline{F_i}(tx_i)\right)}{\overline{F}(t \mathbf{x})} \\
& \leq   \sum_{i=1}^d\frac{\left(\overline{F_i}(tx_i -a_i)
-\overline{F_i}(tx_i)\right)}{\overline{F_i}(tx_i)} \\
& < d \epsilon.
\end{align*}
Letting $\epsilon \to 0$ gives the desired result.
\end{proof}

\begin{remark}
It is worth noting that the above statement and Corollary 11 in
\cite{omey:2006} show that for any probability distribution $F$
supported by the positive quadrant in $\bbr^d$, such that the marginal
distribution $F_i$ of $F$ is subexponential for every $i=1,\ldots, d$,
we have, for all $\mathbf{a}\in [0,\infty)^d$,
$\mathbf{x} \in (0,\infty)^d$ and  $n\geq 1$,
\begin{align} \label{e:omey.n}
\lim_{t \to \infty} \frac{\overline{F^{*n}}(t \mathbf{x} - \mathbf{a})}{\overline{F}(t \mathbf{x})} = n.
\end{align}
 \end{remark}

Using \eqref{e:omey.S} as a definition of multivariate
subexponentiality is, therefore, equivalent to merely requiring
one-dimensional subexponentiality for each marginal distribution. Such
requirement, in particular, cannot guarantee one-dimensional
subexponentiality of the linear combinations, as we have seen in
Example \ref{crnonlin}. In fact, it was shown in \cite{leslie:1989}
that even the sum of independent random variables with subexponential
distributions does not need to have a subexponential distribution.

\section{Multivariate Subexponential Distributions}
\label{sec:our.def}

In this section we introduce a new notion of a multivariate
subexponential distribution. We approach the task with the
multivariate ruin problem in mind. We start with a family
$\mathcal{R}$  of open sets in $\mathbb{R}^d$. Recall that a subset
$A$ of $\bbr^d$ is increasing if $\bx\in A$ and $\ba \in [0,\infty)^d$
imply $\bx+\ba\in A$. Let
\begin{align} \label{sets}
\mathcal{R} = \{ A \subset \mathbb{R}^d:\,
  A~\text{open, increasing},~A^c~\text{convex},~\mathbf{0} \notin
  \overline{A} \}.
\end{align}

\begin{remark} \label{setbasic}
Note that $\calR$ is a cone with respect to the multiplication by
positive scalars. That is,
if $A\in\mathcal{R}$, then $uA \in \mathcal{R}$ for any $u >
0$. Further, half-spaces of the form
\begin{equation} \label{e:H}
H = \left\{ \mathbf{x}:\,  a_1 x_1 + \dots + a_d x_d > b   \right\},
\ b>0, \, a_1,  \dots, a_d \geq 0 \  \text{with} \ a_1+\ldots+a_d=1
\end{equation}
are members of $\mathcal{R}$.
\end{remark}

\begin{remark} \label{trans}
We can write a set $A \in \mathcal{R}$ (in a non-unique way) as $A =
\mathbf{b} + G$, with $\mathbf{b} \in (0,\infty)^d$ and $\mathbf{0}
\in \partial G$ (with $\partial G$ being the boundary of
$G$.) It is clear that the set $G$ is then also increasing. We will
adopt this notation in some of the proofs to follow.
\end{remark}

To see a connection with the multivariate ruin problem, imagine that
for a
fixed set $A\in\calR$ we view $A$ as the ``ruin set'' in the sense
that if, at any time, the excess of claim amounts over the premia
falls in $A$, then the insurance company is ruined. Note that, in the
one-dimensional situation, all sets in $\calR$ are of the form
$A=(u,\infty)$ with $u>0$, so the ruin corresponds to the  excess of
claim amounts over the premia being over the initial capital $u$. The
different shapes of sets in $\calR$ can be viewed as allowing
different interactions between multiple lines of business. For
example, choosing $A$ of the form
$$
A=\left\{ \bx:\, x_i>u_i \ \text{for some $i=1,\ldots,d$}\right\}, \ \
  u_1,\ldots, u_d>0
$$
corresponds to completely separate lines of business, where a ruin of
one line of business causes the ruin of the company. On the other
hand, using as $A$ a half-space of the form \eqref{e:H} corresponds
to the situation where there is a single overall initial capital $b$
and the proportion of $a_i$ in a shortfall in the $i$th line of
business is charged to the overall capital $b$. The connections to the
ruin problem are discussed more thoroughly in Section \ref{sec:ruin}.

\medskip

Before we introduce our notion of multivariate subexponentiality, we
collect, in the following lemma, certain facts about the family $\calR$.

\begin{lemma} \label{basic}
Let $A \in \mathcal{R}$.
\begin{enumerate}
    \item [(a)] If $G= A-\mathbf{b}$  for some $\bb\in  \partial A$, then
      $G^c\supset (-\infty,0]^d$.
    \item [(b)] If $u_1 > u_2 > 0$ then $u_1 A \subset u_2 A$.
    \item  [(c)] There is a set of vectors $I_A \subset \mathbb{R}^d$ such
      that
$$
A=\left\{ \bx\in\bbr^d:\, \bp^T\bx>1 \ \text{for some}\,  \ \bp\in
  I_A\right\}\,.
$$
\end{enumerate}
\end{lemma}

\begin{proof}
(a)  Since $G^c$ is closed, it contains the origin. Since $G$ is
increasing, $G^c$ contains the entire quadrant $(-\infty,0]^d$.

(b) This is an immediate consequence of the fact that $A^c$ is convex and $\mathbf{0} \in A^c$.

(c) Let $\mathbf{x}_0 \in \partial A$. Since  $A^c$ is convex, the
supporting hyperplane theorem (see  e.g. Corollary 11.6.2 in
\cite{rockafellar:2015})  tells us that there exists a (not
necessarily unique)  nonzero vector $\mathbf{p}_{\mathbf{x}_0}$ such
that $\mathbf{p}_{\mathbf{x}_0}^T \mathbf{x} \leq
\mathbf{p}_{\mathbf{x}_0}^T \mathbf{x}_0$ for all $\mathbf{x} \in
A^c$. Since ${\bf 0}\in A^c$, we must have
$\mathbf{p}_{\mathbf{x}_0}^T \mathbf{x}_0\geq 0$. Since $A$ is
increasing, the case $\mathbf{p}_{\mathbf{x}_0}^T \mathbf{x}_0= 0$ is
impossible, so $\mathbf{p}_{\mathbf{x}_0}^T \mathbf{x}_0>0$.

We scale each $\mathbf{p}_{\mathbf{x}_0}$ so that
$\mathbf{p}_{\mathbf{x}_0}^T \mathbf{x}_0 = 1$. Let $I_A$ be the set
of all such $\mathbf{p}_{\mathbf{x}_0}$ for all $\mathbf{x}_0
\in \partial A$. Since a closed convex set equals the intersection of
the half-spaces bounded by
its supporting hyperplanes (see  e.g. Corollary 11.5.1 in
\cite{rockafellar:2015}), the collection $I_A$ has the required
properties.
\end{proof}

\begin{remark} \label{scale}
It is clear that, once we have chosen a collection $I_A$ for some
$A\in\calR$, for any $u>0$ we can use $I_A/u$ as $I_{uA}$.
\end{remark}

We are now ready to define multivariate subexponentiality.
Let $F$ be a probability distribution on
$\mathbb{R}^d$ supported by $[0,\infty)^d$. For a fixed $A \in
\mathcal{R}$ it follows from part (b) of Lemma \ref{basic} that the
function on $[0,\infty)$ defined by
$$
F_A(t) = 1-F(tA), \ t\geq 0\,,
$$
is a probability distribution function on $[0,\infty)$.

\begin{definition} \label{multsubdef}
For any $A \in \mathcal{R}$, we say that $F \in \mathscr{S}_A$ if $F_A
\in \mathscr{S}$, and we write $\mathscr{S}_\mathcal{R} := \cap_{A \in
  \mathcal{R}} \mathscr{S}_A$.
\end{definition}

We view the class $\scrS_\calR$ as the class of subexponential
distributions. However, for some applications we can use a larger
class, such as $\scrS_A$ for a fixed $A\in\calR$, or the intersection
of such classes over a subset of $\calR$.

\medskip

Note that by Remark \ref{setbasic}, if $\mathbf{X}$ is a random vector
in $\bbr^d$ whose distribution is in $\mathscr{S}_\mathcal{R}$, then
all non-degenerate linear combinations of the components of
$\mathbf{X}$ with
nonnegative coefficients have one-dimensional subexponential
distributions. More generally, we have the following stability
property. We say that a linear transformation $T:\, \bbr^d\to\bbr^k$
is increasing if $T\bx\in [0,\infty)^k$ for any $\bx\in [0,\infty)^d$.
\begin{proposition} \label{pr:stability}
Let $T:\, \bbr^d\to\bbr^k$ be a linear increasing transformation. If
$\BX$ is a random vector in $\bbr^d$ whose distribution is in
$\mathscr{S}_\mathcal{R}$ (in $\bbr^d$), then the same is true (in
$\bbr^k$) for the distribution of the random vector $T\BX$.
\end{proposition}
\begin{proof}
The statement follows from the easily checked fact that for any
$A\in\calR$ in $\bbr^k$, the set $T^{-1}A$ is in $\calR$ in $\bbr^d$.
\end{proof}

The following lemma is useful. Its argument uses the fact that for any
random vector $X$ on $\mathbb{R}^d$ and $A \in \mathcal{R}$, we can
write, for any $u > 0$, the event
$\left\{ \mathbf{X} \in uA \right\}$ as $\left\{ \sup_{\mathbf{p} \in
    I_A} \mathbf{p}^T \mathbf{X} > u \right\}$; see Lemma \ref{basic}
and Remark \ref{scale}.

\begin{lemma} \label{upbound}
For any $A \in \mathcal{R}$ and $n \geq 1$,
\begin{align}
\overline{(F_A)^{*n}}(t) \geq F^{*n}(tA).
\end{align}
\end{lemma}

\begin{proof}
Let $\mathbf{X}^{(1)}, \dots, \mathbf{X}^{(n)}$ be independent random
vectors with distribution $F$. Let $Y_1, \dots, Y_n$ be
one-dimensional random variables defined by
\begin{align*}
Y_i = \sup \{ u:\,  \mathbf{X}^{(i)}\in uA \} = \sup_{\mathbf{p} \in I_A}
  \mathbf{p}^T \mathbf{X}^{(i)},  \ i=1,\ldots, d\,,
\end{align*}
see Remark \ref{scale}.
Note that by part (b) of Lemma \ref{basic},
\begin{align*}
P(Y_i > t) & = P( \mathbf{X}^{(i)} \in tA) \\
& = F(tA) = \overline{F_A}(t).
\end{align*}

Hence it follows that
\begin{align*}
F^{*n}(tA) & = P(\mathbf{X}^{(1)} + \dots + \mathbf{X}^{(n)} \in tA) \\
& = P\bigl(\sup_{\mathbf{p} \in I_A} \mathbf{p}^T (\mathbf{X}^{(1)} + \dots +
  \mathbf{X}^{(n)}) > t\bigr) \\
&\leq  P\bigl(\sup_{\mathbf{p} \in I_A} \mathbf{p}^T \mathbf{X}^{(1)}
  + \dots +\sup_{\mathbf{p} \in I_A} \mathbf{p}^T \mathbf{X}^{(n)}>
  t\bigr) \\
&= P(Y_1+\ldots +Y_n>t) \\
& = \overline{(F_A)^{*n}}(t)\,,
\end{align*}
as required.
\end{proof}

In spite of this lemma, the two probabilities are asymptotically
equivalent.
\begin{corollary} \label{convtail}
$A \in \mathcal{R}$. Let $\mathbf{X}, \mathbf{X}^{(1)}, \dots,
\mathbf{X}^{(n)}$ be independent random vectors with distribution
$F$. If $F \in \mathscr{S}_A$ for some $A \in \mathcal{R}$,
then for all $n\geq 1$,
\begin{align} \label{e:weaker.ass}
\lim_{u \to \infty} \frac{P(\mathbf{X}^{(1)} + \dots +
  \mathbf{X}^{(n)} \in uA)}{P(\mathbf{X} \in uA)} = n.
\end{align}
\end{corollary}
\begin{proof}
It follows from Lemma \ref{upbound} that only an asymptotic lower
bound needs to be established. However,
since $\mathbf{X}^{(1)}, \dots, \mathbf{X}^{(n)}$ are all nonnegative,
and $A$ is an increasing set,  it must be that if $\mathbf{X}^{(1)} + \dots
+ \mathbf{X}^{(n)} \in uA^c$, then each $\mathbf{X}^{(1)}, \dots, \mathbf{X}^{(n)} \in
uA^c$. Therefore,
\begin{align*}
P(\mathbf{X}^{(1)} + \dots + \mathbf{X}^{(n)} \in uA^c) & \leq
 P(\mathbf{X}^{(1)}, \dots, \mathbf{X}^{(n)} \in uA^c) \\
& = P(\mathbf{X} \in uA^c)^n.
\end{align*}
It follows that
\begin{align*}
\liminf_{n\to\infty} \frac{P(\mathbf{X}^{(1)} + \dots + \mathbf{X}^{(n)} \in
  uA)}{P(\mathbf{X} \in uA)} & \geq \liminf_{n\to\infty} \frac{1 - P(\mathbf{X} \in
 uA^c)^n}{P(\mathbf{X} \in uA)}  = n\,,
\end{align*}
as required.
\end{proof}

\begin{remark} \label{nonequiv}
We note at this point that the assumption $F \in \mathscr{S}_A$ is NOT
equivalent to the assumption that \eqref{e:weaker.ass} holds for all
$n$. In fact, the latter assumption is weaker. To see that, consider
the following example. Let
$X$ and $Y$ be two independent nonnegative one-dimensional random
variables with subexponential distributions, such that $X + Y$ is not
subexponential; recall that such random variables exist, see
\cite{leslie:1989}. We construct a bivariate random vector
$\mathbf{Z}$ by taking a Bernoulli $(1/2)$ random variable $B$
independent of $X$ and $Y$ ans setting  $\mathbf{Z} = (X, 0)$ if $B=0$
and   $\mathbf{Z} = (0, Y)$ if $B=1$.  Let $A = \{ (x, y):\, \max(x, y) >
1 \}$. Since the marginal distributions of the bivariate distribution
of $\BZ$ are, obviously, subexponential, we see by \eqref{e:omey.n}
that \eqref{e:weaker.ass} holds for all $n \geq
1$. However, for $u>0$,
$$
\overline{F_A}(u) = \frac{1}{2} P(X > u) +
\frac{1}{2} P(Y > u)\,,
$$
so the distribution $F_A$  is a mixture of the distributions of $X$
and $Y$. By Theorem 2 of \cite{embrechts:goldie:1980}, any non-trivial
mixture of the distributions of $X$ and $Y$ is subexponential if any
only if their convolution is. Since, by construction, that convolution
is not subexponential, we conclude that  $F_A \notin \mathscr{S}$ and $F
\notin \mathscr{S}_A$.
\end{remark}

In the next proposition we check that the basic properties of
one-dimensional subexponential distributions extend to the
multivariate case.

\begin{proposition} \label{properties}
Let $A \in \mathcal{R}$ and $F \in \mathscr{S}_A$.
\begin{enumerate}
    \item [(a)] If $G$ is a  distribution on
      $\mathbb{R}^d$ supported by $[0,\infty)^d$, such that
        \begin{align*}
        \lim_{u \to \infty} \frac{F(uA)}{G(uA)} = c > 0,
        \end{align*}
        then $G \in \mathscr{S}_A$.
    \item [(b)] For any $\mathbf{a} \in \mathbb{R}^d$,
        \begin{align}
        \lim_{u \to \infty} \frac{F(uA + \mathbf{a})}{F(uA)} = 1.
        \end{align}
    \item [(c)] Let $\mathbf{X}, \mathbf{X}^{(1)}1, \dots,
      \mathbf{X}^{(n)}$ be
      independent random vectors with distribution $F$. For any
      $\epsilon > 0$, there exists $K > 0$ such that for  all $u > 0$
      and  $n \geq 1$,
        \begin{align}
        \frac{P(\mathbf{X}^{(1)} + \dots + \mathbf{X}^{(n)} \in
          uA)}{P(\mathbf{X} \in uA)} < K (1 + \epsilon)^n \,.
        \end{align}
  \end{enumerate}
\end{proposition}

\begin{proof}
(a) This is an immediate consequence of the univariate
subexponentiality of $F_A$ and the corresponding property of
one-dimensional subexponential distributions; see e.g. Lemma 4 in
\cite{embrechts:goldie:veraverbeke:1979}.

(b) Write $A = \mathbf{b} + G$ as in Remark \ref{trans}. As in part
(b) of Lemma \ref{basic}, we have $u_1 G \subset u_2 G$ if  $u_1 >
u_2 > 0$. Since $G$ is an increasing set, it follows that
there exists $u_1 > 0$ such that for all   $u>u_1$ we have $(u +
u_1)A \subset uA  + \mathbf{a} \subset (u - u_1)A$. Therefore,
\begin{align*}
\overline{F_A}(u + u_1) & = F((u + u_1)A) \\
& \leq F(uA  + \mathbf{a}) \\
& \leq F((u - u_1)A) = \overline{F_A}(u - u_1)\,,
\end{align*}
and the claim follows from the one-dimensional long tail property of
$F_A$.

(c) The claim follows from Lemma \ref{upbound} and the corresponding
one-dimensional bound; see e.g. Lemma 3 in
\cite{embrechts:goldie:veraverbeke:1979}.
\end{proof}

\begin{remark} \label{alldist}
In our Definition \ref{multsubdef} of multivariate subexponentiality
one can drop the assumption that a distribution is supported by
$[0,\infty)^d$.  We can check that both Corollary \ref{convtail} and
Proposition \ref{properties} remain true in this extended case.
 \end{remark}

Our next step is to show that multivariate regular varying
distributions fall within the class $\scrS_\calR$ of multivariate
subexponential distributions. The definition of non-standard
multivariate regular
variation for distribution supported by $[0,\infty)^d$ was given in
\eqref{e:regvar.d}. Presently we would only consider the standard
multivariate regular variation, but allow distributions
not necessarily restricted to the first quadrant. In this case one
assumes that there is a non-zero Radon measure $\mu$  on
$[-\infty,\infty]^d\setminus \{{\bf 0}\}$, charging only finite
points, and a function $b$ on $(0,\infty)$ increasing
to infinity, such that
\begin{equation} \label{e:regvar.d.st}
tF\bigl( b(t)\cdot\bigr)\vconv \mu
\end{equation}
vaguely  on $[-\infty,\infty]^d \setminus \{{\bf 0}\}$.
Recall that the measure $\mu$ is called the tail measure
of $\BX$; it has automatically a scaling property: for some
$\alpha>0$,
  $\mu(uA) = u^{- \alpha} \mu(A)$ for every $u > 0$ and every Borel
  set $A \in \mathbb{R}^d$, and the function $b$ in
  \eqref{e:regvar.d.st} is regularly varying with exponent $1/\alpha$;
  see \cite{resnick:2007}.
We say that $F$ (and $\BX$) are regularly varying
  with exponent $\alpha$ and  use the notation $F\in
  MRV(\alpha,\mu)$.
\begin{proposition} \label{mrvprop}
$MRV(\alpha,\mu)\subset \mathscr{S}_{\mathcal{R}}$.
\end{proposition}
\begin{proof}
We start by showing that for any $A \in \mathcal{R}$,
$\mu(\partial A) = 0$. Since for any $u>0$,
$$
\mu(\partial (u A)) = \mu(u\partial A) = u^{-\alpha}\mu(\partial A)\,,
$$
 it is enough to show that for any
$u > 1$, $\partial (uA) \cap \partial A = \emptyset$ (indeed,
$\mu(\partial A) > 0$ would then imply existence of uncountably many
disjoint sets of positive measure).

Suppose, to the contrary, that $\partial (uA) \cap \partial A \neq
\emptyset$, and let $\mathbf{x} \in \partial (uA) \cap \partial
A$. The set $I_A$ in part (c) of Lemma \ref{basic}, has, by
construction, the property that $u^{-1}\bx$, as an element of
$u^{-1}\partial (u A)=\partial A$,
satisfies $\mathbf{p}^Tu^{-1}
\mathbf{ x} = 1$ for some $\mathbf{p} \in I_{A}$. But then
$\mathbf{p}^T\mathbf{ x} = u>1$, which says that $\bx$ is in $A$,
rather than in $\partial A$, which is a subset of $A^c$.

It follows from \eqref{e:regvar.d.st} that for any set $A\in \calR$,
\begin{align*}
tP\bigl(\mathbf{X} \in b(t) A\bigr) \to \mu(A) \in (0,\infty)
\end{align*}
as $t\to\infty$. Since the function $b$ is regularly varying with
exponent $1/\alpha$, we immediately conclude that the distribution
function $F_A$ has a regularly varying tail, hence $F_A$ is
subexponential. Because $A \in \mathcal{R}$ is arbitrary, it follows
that $F \in \cap_{A \in \mathcal{R}} \mathscr{S}_A =
\mathscr{S}_\mathcal{R}$.
\end{proof}

We proceed with clarifying the relation between the class
$\mathscr{S}_\mathcal{R}$ we have introduced in this section and the classes
$\mathscr{S}(\nu; \mathbf{b})$ and $S(\bbr^d)$ of Section
\ref{sec:previous}. We will also provide several examples of
distributions that belong to $\mathscr{S}_\mathcal{R}$, as well as
sufficient conditions for a distribution to be a member of
$\mathscr{S}_\mathcal{R}$.

  Example \ref{crnonlin}, combined with Proposition
  \ref{pr:stability},
  show that neither $\mathscr{S}(\nu; \mathbf{b})$ nor $S(\bbr^d)$ are
 subsets of   $\mathscr{S}_\calR$. We will present an example to show
 that  $\mathscr{S}_\calR   \not\subset \mathscr{S}(\nu; \mathbf{b})$.

We start with presenting
  a sufficient condition for a distribution $F$ to be a member of
  $\mathscr{S}_\mathcal{R}$. We assume for the moment that $F$ is
  supported by $[0,\infty)^d$.

Let $\mathbf{X} \sim F$ be a nonnegative random vector on
$\mathbb{R}^d$ such that $P(\mathbf{X} = \mathbf{0}) = 0$. Denote the
$L_1$ norm of $\BX$ by
\begin{align}
W = || \mathbf{X} ||_1 = \sum_{i = 1}^d X_i\,,
\end{align}
and the projection of $\BX$ onto the $d$-dimensional unit simplex
$\Delta_d$ by
\begin{align}
I = \frac{\mathbf{X}}{|| \mathbf{X} ||_1} = \frac{\mathbf{X}}{W} \in
  \Delta_d\,.
\end{align}

Let $\nu$ be the distribution of $I$ over $\Delta_d$, and let
$(F_{\pmb \theta})_{\pmb \theta \in \Delta_d}$ be a set of regular
conditional distributions of $W$ given $I$.
Notice that,  if  the law $F$ of $\mathbf{X}$ in $\mathbb{R}^d$ has a
density $f$ with respect to the $d$-dimensional Lebesgue measure, then
a version of $(F_{\pmb \theta})_{\pmb \theta \in \Delta_d}$ has
densities with respect to the one-dimensional Lebesgue measure, given by
\begin{align} \label{rcddens}
f_{\pmb \theta} (w) = \frac{w^{d - 1} f(w\, \pmb \theta)}{\int_0^\infty
  u^{d - 1} f(u \,\pmb \theta) \, du},~w > 0.
\end{align}

\begin{proposition} \label{sdom}
Suppose $\mathbf{X}$ is a  random vector
on $\mathbb{R}^d$
with distribution $F$, supported by $[0,\infty)^d$,
such that $P ( \mathbf{X} = \mathbf{0}
) = 0$. Suppose the marginal distributions $F_i$, $i=1,\ldots, d$  have
dominated varying tails. Further, assume that there
is a set of regular conditional distributions $(F_{\pmb \theta})_{\pmb
  \theta \in   \Delta_d}$ of $W$ given $I$ such that $F_{\pmb \theta}
\in \mathscr{L}$ for each $\pmb \theta \in
  \Delta_d$  and for some $C, t_0 > 0$,
\begin{align} \label{e:unif.double}
\frac{\overline{F_{\pmb \theta_1}}(2 t)}{\overline{F_{\pmb
  \theta_2}}(t)} \leq C
\end{align}
for all $t > t_0$ and for all $\pmb \theta_1, \pmb \theta_2 \in
\Delta_d$. Then $F \in \mathscr{S}_\mathcal{R}$.
\end{proposition}

\begin{proof}
Let $A \in \mathcal{R}$ be fixed. Since each of the marginal
distributions have
dominated varying tails, it follows that $F_A$ also has a
dominated varying tail. Since $\mathscr{L} \cap \mathscr{D} \subset
\mathscr{S}$, it suffices to show that $F_A \in \mathscr{L}$.

For $\pmb \theta \in \Delta_d$, let
\begin{align} \label{height}
h_{\pmb \theta} = \inf \left\{ w > 0:\,  w \pmb \theta \in A
  \right\}>0\,,
\end{align}
Note that $h_{\pmb \theta}$ is bounded away from 0. Further,
by convexity of $A^c$, $h({\bf e}^{(i)})<\infty$ for at
least one coordinate vector ${\bf e}^{(i)}$, $i=1,\ldots, d$. Since
the dominated variation of the marginal tails implies, in particular,
that each coordinate of the vector $\BX$ is positive with positive
probability, we conclude that
$$
\nu\bigl\{ {\pmb \theta} \in \Delta_d:\, h_{\pmb
  \theta}<\infty\bigr\}>0\,.
$$
We conclude that there is $M>0$ and a measurable set $B\subset
\Delta_d$ with $\delta:= \nu(B)>0$, such that
$$
1/M\leq h_{\pmb \theta}\leq M\ \ \text{for all ${\pmb \theta} \in B$.}
$$

Note that for $t > 0$,
\begin{align} \label{rcd}
\overline{F}_A(t) = \int_{\Delta_d} \overline{F_{\pmb \theta}}(t
  h_{\pmb \theta}) \, \nu(d \pmb \theta).
\end{align}
Therefore,
\begin{align} \label{ltdif}
\frac{\overline{F_A(t)} - \overline{F_A(t + 1)}}{\overline{F_A}(t)} =
  \frac{\int_{\Delta_d} (\overline{F_{\pmb \theta}}(t h_{\pmb
  \theta}) - \overline{F_{\pmb \theta}}((t + 1)  h_{\pmb
  \theta})) \, \nu(d \pmb \theta)}{\int_{\Delta_d} \overline{F_{\pmb
  \theta}}(t h_{\pmb \theta}) \, \nu(d \pmb \theta)}\,,
\end{align}
and we wish to show that this quantity goes to $0$ as $t \to \infty$.

By the assumptions, for any fixed $\pmb \theta$,
 $F_{\pmb \theta} \in
\mathscr{L}$, hence for any fixed $\pmb \theta$ such that $h_{\pmb
\theta}<\infty$,
\begin{align*}
\lim_{t \to \infty} \frac{\overline{F_{\pmb \theta}}(t h_{\pmb \theta}) - \overline{F_{\pmb \theta}}((t + 1)  h_{\pmb \theta})}{\overline{F_{\pmb \theta}}(t  h_{\pmb \theta})} = 0.
\end{align*}
Therefore, for a given $\epsilon > 0$, there exists $t_\epsilon > 0$
such that, for all $t > t_\epsilon$, $\nu(S_{t, \epsilon}) <
\epsilon$, where
\begin{align*}
S_{t, \epsilon} = \left\{ \pmb \theta \in \Delta_d:\, h_{\pmb
\theta}<\infty \ \text{and} \
  \frac{\overline{F_{\pmb \theta}}(t  h_{\pmb \theta}) -
  \overline{F_{\pmb \theta}}((t + 1)  h_{\pmb
  \theta})}{\overline{F_{\pmb \theta}}(t  h_{\pmb \theta})} >
  \epsilon \right\}.
\end{align*}

Let $\epsilon <(\delta/2)^2$. Then $\nu\bigl( B\cap S_{t,
  \epsilon}^c\bigr)>\bigl(\nu(S_{t, \epsilon})\bigr)^{1/2}$. By the
definition of the set $B$ and by \eqref{e:unif.double},
 for some $C_1, \, \tilde{t}_0 > 0$
\begin{align*}
\frac{\overline{F_{\pmb \theta_1}}(t \cdot h_{\pmb \theta_1})}{\overline{F_{\pmb \theta_2}}(t \cdot h_{\pmb \theta_2})} \leq C_1
\end{align*}
for any $\pmb \theta_1\in S_{t, \epsilon}$ and any $\pmb \theta_2 \in
B\cap S_{t,   \epsilon}^c$, for all $t > \tilde{t}_0$. Therefore, for
$t > t_\epsilon + \tilde{t}_0$,
\begin{align*}
\int_{S_{t, \epsilon}} \overline{F_{\pmb \theta}}(t  h_{\pmb \theta}) - \overline{F_{\pmb \theta}}((t + 1)  h_{\pmb \theta}) \, \nu(d \pmb \theta) & \leq  \int_{S_{t, \epsilon}} \overline{F_{\pmb \theta}}(t  h_{\pmb \theta}) \, \nu(d \pmb \theta) \\
& < \frac{\nu(S_{t, \epsilon})}{\nu(B\cap S_{t, \epsilon})^c} C_1
  \int_{B\cap S_{t, \epsilon}^c} \overline{F_{\pmb \theta}}(t \cdot h_{\pmb \theta}) \, \mu(d \pmb \theta) \\
& <  \epsilon^{1/2} C_1 \int_{\Delta_d} \overline{F_{\pmb
  \theta}}(t \cdot h_{\pmb \theta}) \, \nu(d \pmb \theta).
\end{align*}
Hence, for $t > t_\epsilon + \tilde{t}_0$, the quantity in
(\ref{ltdif}) is bounded above by $\epsilon +\epsilon^{1/2}
C_1$. Letting $\epsilon \searrow 0$ gives us the desired result.
\end{proof}

We are now ready to give an example showing that
$\mathscr{S}_\mathcal{R} \not\subset \mathscr{S}(\nu; \mathbf{b})$.
\begin{example} \label{snotincr}
Let $0<|\gamma|\leq 1/2$. It is shown in in \cite{cline:resnick:1992}
  that a legitimate probability distribution $F$, supported by
  $(0,\infty)^2_+$, satisfies
\begin{align}
P(X > x, Y > y) = \frac{1 + \gamma \sin( \log(1 + x + y)) \cos(
  \frac{1}{2} \pi \frac{x - y}{1 + x + y})}{1 + x + y}, \, x,y\geq 0\,.
\end{align}
Then
$$
P(X>x)=P(Y>x) \sim x^{-1} \ \ \text{as $x\to\infty$}\,,
$$
but $F\notin \mathscr{S}(\nu; \mathbf{b})$; see
\cite{cline:resnick:1992}.
Straightforward differentiation gives us the density $f$ of $F$, and
one can check that it satisfies
 \begin{align*}
\frac{2 - 4 \gamma - 3 \gamma \pi - \pi^2 / 4}{(1 + x + y)^3} \leq f(x,y) \leq \frac{2 + 4 \gamma + 3 \gamma \pi + \pi^2 / 4}{(1 + x + y)^3}\,,
\end{align*}
so by \eqref{rcddens}, we have
\begin{align*}
a \frac{w}{(1 + w)^3} \leq f_{\pmb \theta}(w) \leq b \frac{w}{(1 + w)^3},~w > 0,
\end{align*}
for some $0 < a < b < \infty$, independent of $\pmb \theta$. It is
clear that the conditions of Proposition \ref{sdom} are satisfied and,
hence, $F \in \mathscr{S}_\mathcal{R}$.
\end{example}

Proposition \ref{sdom} gives us a way to check that a multivariate
distribution belongs to the class $\mathscr{S}_\mathcal{R}$, but it
only applies to distributions that have,
marginally, dominated varying tails. In the remainder of this section
we provide sufficient conditions for membership in
$\mathscr{S}_\mathcal{R}$ that do not require marginals with dominated
varying tails. We start with a motivating example.

\begin{example} \label{ex:isotropic} [Rotationally invariant case]
Assume that there is a one-dimensional distribution $G$ such that
$\overline{F_{\pmb \theta}} = \overline{G}$ for all $\pmb \theta \in
\Delta_d$.  Let $A\in\calR$, and notice that, in the rotationally
invariant case, a random variable $Y_A$ with distribution $F_A$ can be
written, in law, as
\begin{align} \label{e:product}
Y_A \eid Z  H^{-1},
\end{align}
with $Z$ and $H$ being independent, $Z$ with the distribution $G$, and
$H=h_\Theta$. Here $h$ is defined by \eqref{height}, and $\Theta$
has the law $\nu$ over the simplex $\Delta_d$. Recall that the
function $h$ is bounded away from zero, so that the random variable
$H^{-1}$ is bounded. If $G \in \mathscr{S}$, then  the product in the
right hand side is subexponential by Corollary 2.5 in
  \cite{cline:samorodnitsky:1994}. Hence $F_A \in \mathscr{S}$ for all
  $A\in\calR$,  and so $F \in \mathscr{S}_\mathcal{R}$.
\end{example}

The rotationally invariant case of Example \ref{ex:isotropic}  can be
slightly extended, without much effort, to the case where there is a bounded
positive function $\bigl( a_{\pmb \theta}, \, {\pmb \theta}\in \Delta_d\bigr)$
such that $F_{\pmb   \theta}(\cdot) = G(\cdot / a_{\pmb \theta})$ for
some $G \in \mathscr{S}$. An argument similar to the one in the
example shows that we can still conclude that $F \in \mathscr{S}_\mathcal{R}$.
In order to achieve more than that, we note that the distribution $F_A$ can
be represented, by \eqref{rcd}, as a mixture of scaled regular
conditional distributions. Note also that the product of independent
random variables in \eqref{e:product} is just a special case of that
mixture, to which we have been able to apply  Corollary 2.5 in
  \cite{cline:samorodnitsky:1994}. It is  likely to be possible to
  extend that result to certain mixtures that are more general
  than products of   independent random variables, and thus to obtain
additional criteria for membership in the class
$\mathscr{S}_\mathcal{R}$. We leave serious extensions of this type to future
work. A small extension that still steps away from exact products is
below, and it takes a result in  \cite{cline:samorodnitsky:1994} as
an ingredient.   We formulate the statement in terms of
  the distribution of a random variable that only in a certain
  asymptotic sense looks like a product of independent random
  variables.

\begin{theorem} \label{mixture}
Let $(\Omega_i,{\mathcal F}_i,P_i)$, $i=1,2$ be probability
spaces. Let $Q$ be a random variable defined on the product
probability space. Assume that there are nonnegative random variables
$X_i$, $i=1,2$, defined on $(\Omega_1,{\mathcal F}_1,P_1)$ and
$(\Omega_2,{\mathcal F}_2,P_2)$ correspondingly, such that
$X_1$ has a subexponential distribution $F$, and for some $t_0>0$ and
$C>0$,
\begin{equation} \label{e:almost.prod}
X_1(\omega_1)X_2(\omega_2) -C X_2(\omega_2) \leq Q(\omega_1,\omega_2)
\leq X_1(\omega_1)X_2(\omega_2) +C X_2(\omega_2)
\end{equation}
a.s. on the set $\{ Q(\omega_1,\omega_2)>t_0\}$. Suppose $P(X_2>0)>0$, and let
$G$  be the distribution of $X_2$. Suppose that there is a
function $a: (0, \infty) \to (0, \infty)$, such that
\begin{enumerate}
    \item $a(t) \nearrow \infty$ as $t \to \infty$;
    \item $\frac{t}{a(t)} \nearrow \infty$ as $t \to \infty$;
    \item $\lim_{t \to \infty} \frac{\overline{F}(t -
        a(t))}{\overline{F}(t)} = 1$;
    \item $\lim_{t \to \infty} \frac{\overline{G}(a(t))}{P(X_1 X_2 > t)} = 0$.
\end{enumerate}
Then $Q$ has a subexponential distribution.
\end{theorem}
\begin{proof}
Let $H$ denote the distribution of $X_1X_2$.  It follows by Theorem 2.1 in
\cite{cline:samorodnitsky:1994}  that $H$ is subexponential.  We show that
$P(Q > t) \sim \overline{H}(t)$ as $t\to\infty$. This will imply that
$Q$ has a subexponential distribution.

We start by checking that
\begin{align} \label{e:H.same}
\lim_{t \to \infty} \frac{\overline{H}(t - a(t))}{\overline{H}(t)} = 1,
\ \ \text{implying that} \ \lim_{t \to \infty} \frac{\overline{H}(t +
  a(t))}{\overline{H}(t)} = 1\,,
\end{align}
since $a(t+a(t))\geq a(t)$. To verify the limit, suppose first that
$X_2\geq 1$ a.s., and write
\begin{align*}
& P\bigl( t - a(t)<X_1X_2\leq t\bigr) \leq P_2(X_2 > a(t)) \\
&  \leq \int_{\Omega_2} P_1( t/X_2(\omega_2) -
 a(t)/X_2(\omega_2)<X_1\leq t/X_2(\omega_2) \bigr) \one\bigl(
 X_2(\omega_2) \leq a(t)\bigr)\, P_2(d\omega_2)
\end{align*}
The first term in the right hand side is $o(\overline{H}(t))$ by the
assumption (4), while the same is true for the second term by the
assumption (3), since by the assumption (2), $a(t)/y\leq a(t/y)$ if
$y\geq 1$. This proves \eqref{e:H.same} if $X_2\geq 1$ a.s. and
hence, by scaling, if $X_2\geq \epsilon$ a.s. for some
$\epsilon>0$. An elementary truncation argument then shows that
\eqref{e:H.same}  holds if $P(X_2>0)>0$.

Note that for $t>t_0$,
\begin{align*}
P(Q > t) &\leq P( X_1X_2+CX_2>t) \\
&\leq \overline{G}(a(t)) +  \overline{H}(t-Ca(t))\,.
\end{align*}
This implies that  $\limsup_{t \to \infty} P(Q >
  t)/\overline{H}(t)\leq 1$. The statement $\liminf_{t
  \to \infty} P(Q > t)/\overline{H}(t)\geq 1$ can be shown in a
similar way.
\end{proof}

Despite a limited scope of the extension given in Theorem
\ref{mixture}, it allows one to construct a number of examples of
multivariate distributions in $\mathscr{S}_\mathcal{R}$ by choosing,
for example,
$\Omega_2=\Delta_d$ and $X_2({\pmb \theta})=1/h({\pmb \theta})$, $
{\pmb \theta}\in \Delta_d$, and selecting a function $Q$ to model
additional randomness in the radial direction.

\bigskip

\section{Ruin Probabilities}
\label{sec:ruin}

As mentioned in the introduction, the notion of subexponentiality we
introduced in Section \ref{sec:our.def} was designed with insurance
applications in mind. In this section we describe such an application
more explicitly.

Consider a renewal model for the reserves of an insurance company with
$d$ lines of business. Suppose that claims arrive according to a renewal
process $(N_t)_{t \geq 0}$ given by $N_t = \sup\{n \geq 1: T_n \leq
  t \} $. The arrival times $(T_n)$ form a renewal sequence
\begin{align}
T_0 = 0,~~~T_n = Y_1 + \dots + Y_n~\text{for}~n \geq 1,
\end{align}
where the interarrival times $(Y_i)_{i \geq 1}$ form a sequence of
independent and identically distributed positive random variables. We
will call a generic interarrival time $Y$.
At the arrival time $T_i$ a random vector-valued claim size
$\BX^{(i)}=\bigl(X^{(i)}_1,\ldots, X^{(i)}_d\bigr)$ is incurred, so that the
part of the claim going to the $j$th line of business is $X^{(i)}_j$.
We assume that the  claim sizes $(\BX^{(i)})$ are i.i.d. random
vectors with a finite mean, and we denote their common law by $F$. We
assume further that
the claim size process is independent of the renewal
process of the claim arrivals. The $j$th line of business collects
premium at the rate of  $p^j$ per unit of time.
Let $\mathbf{p}$ be the vector of the premium rates, and
 $\mathbf{X}$  a generic random vector of  claim sizes.

Suppose that the
company has an initial buffer capital of $u$, out of which the amount
of $u b_j$ is allocated to the $j$th line of business, $j=1,2,\ldots,
d$. Here $b_1,\ldots, b_d$ are positive numbers,  $b_1 +
\dots + b_d = 1$. Then $u \mathbf{b}$ denotes the vector for the
initial capital buffer allocation. With the above notation, the claim
surplus process $(\mathbf{S}_t)_{t \geq 0}$ and the risk reserve
process $(\mathbf{R}_t)_{t \geq 0}$ are given by
\begin{align*}
\mathbf{S}_t = \sum_{i = 1}^{N_t} \mathbf{X}^{(i)} - t
\mathbf{p},~~~\mathbf{R}_t = u \mathbf{b} - \mathbf{S}_t = u
\mathbf{b} + t \mathbf{p} - \sum_{i = 1}^{N_t} \mathbf{X}^{(i)}, \ t\geq 0\,.
\end{align*}

The company becomes insolvent (ruined) when the risk reserve process
hits a certain ruin set $L \subset \mathbb{R}^d$.  Equivalently, ruin
occurs when the claim surplus process enters the set $u\bb-L$. We will
assume that the ruin set satisfies the following condition.
\begin{assumption} \label{ass:ruin.set}
The ruin set is an open decreasing set such that $\mathbf{0}
\in \partial L$, satisfying  $L=uL$ for $u>0$, and such that $L^c$ is
convex.
\end{assumption}

Note that this assumption means that the ruin occurs when the claim
surplus process enters the set $uA$, with $A=\bb-L\in \calR$, as
defined in Section \ref{sec:our.def}. In fact, the ruin set $L$ can be
viewed as being of the form $-G$, as defined in Remark \ref{trans}.
Examples of such ruin sets are, of course, the sets
$$
L = \bigl\{ \bx:\, x_j<0 \ \text{for some} \ j=1,\ldots, d\bigr\}
\ \ \text{and} \ \
L = \bigl\{ \bx:\, x_1+\ldots +x_d<0\bigr\}\,,
$$
discussed in Section \ref{sec:our.def}. A general framework was
proposed in \cite{hult:lindskog:2006a}.  In this framework capital can
be transferred between different business lines, but the transfers
incur costs, and the solvency set has the form
\begin{align} \label{solvency}
L^c = \left\{ \mathbf{x}:\, \mathbf{x} = \sum_{i \neq j}
  v_{ij}(\pi_{ij} \mathbf{e}^i - \mathbf{e}^j) + \sum_{i = 1}^d w_i
  \mathbf{e}^i, \, v_{ij}\geq 0,\,  w_i \geq 0 \right\},
\end{align}
where $\mathbf{e}^1, \dots, \mathbf{e}^d$ are the standard basis
vectors, and $\Pi = ( \pi_{ij} )^d_{i,j = 1}$ is a matrix satisfying
\begin{enumerate}
    \item [(i)] $\pi_{ij} \geq 1$ for $i,j \in \left\{1, \dots, d \right\}$,
    \item [(ii)] $\pi_{ii} = 1$ for $i \in \left\{1, \dots, d \right\}$,
    \item [(iii)] $\pi_{ij} \leq \pi_{ik} \pi_{kj}$ for $i, j, k \in \left\{ 1, \dots, d \right\}$.
\end{enumerate}
In the financial literature, a matrix satisfying the above constraints
is called a bid-ask matrix. In our context, the entry $\pi_{ij}$ can
be interpreted as the amount of capital that needs to be taken from
business line $i$ in order to transfer $1$ unit of capital to
business line $j$.

We note that each of the above ruin sets is a cone, i.e. it satisfies
$L=uL$ for $u>0$, as assumed in Assumption \ref{ass:ruin.set}.

We maintain the notation $A=\bb-L\in \calR$. Note that we can write
the ruin probability  as
\begin{align} \label{ruinprob}
\psi_{\mathbf{b}, L}(u) & = P(\mathbf{R}_t \in L~\text{for some}~ t
                          \geq 0) \\
& = P\left(\sum_{i = 1}^{n} \mathbf{X}^{(i)} - Y_i \mathbf{p} \in uA~\text{for
  some}~ n \geq 1\right)  \notag \\
& = P\left(\sum_{i = 1}^{n} \mathbf{Z}^{(i)}  \in uA~\text{for
  some}~ n \geq 1\right)\,, \notag
\end{align}
where $\mathbf{Z}^{(i)} = \mathbf{X}^{(i)} - Y_i \mathbf{p}$,
$i=1,2,\ldots$. We let $\mathbf{Z}$ denote a generic element of the
sequence $(\mathbf{Z}^{(i)})_{i \geq 1}$. We will assume a
positive safety loading, an assumption that takes now the form
$$
\mathbf{c} = -\mathbb{E}[\mathbf{Z}] > \mathbf{0}\,,
$$
see e.g. \cite{asmussen:2000}. The assumption of the finite mean for
the claim sizes implies that
$$
\theta:=\int_0^\infty F\Bigl(  [0,\infty)^d
+ v \mathbf{c}\Bigr)   \, dv <\infty\,,
$$
and we can defined a probability measure on $\bbr^d$, supported by
$[0,\infty)^d$, by
\begin{equation} \label{e:F.I}
F^I(\cdot) =\frac{1}{\theta} \int_0^\infty F(  \cdot
+ v \mathbf{c})   \, dv\,.
\end{equation}
Denote
\begin{align} \label{inttail}
H(u) =   \int_0^\infty F(uA + v \mathbf{c}), \ u>0\,.
\end{align}

The following is the main result of this section.
\begin{theorem} \label{ruintail}
Suppose that the law $F^I$  is in $\mathscr{S}_A$. Then the ruin
probability $\psi_{\mathbf{b}, L}$ satisfies
\begin{align} \label{e:ruin.asymp}
\lim_{u \to \infty} \frac{\psi_{\mathbf{b}, L}(u)}{H(u)} = 1.
\end{align}
\end{theorem}

\begin{remark} \label{uniruin}
Notice, for comparison, that in the univariate case, with the ruin set
 $L = ( -\infty, 0)$ (and $b = 1$) we have $A = (1,
\infty)$, and
\begin{align*}
H(u) = \int_0^\infty \overline{F}(u + vc) \, dv =
  \frac{1}{c} \int_u^\infty \overline{F}(v) \, dv\,.
\end{align*}
In this case the statement \eqref{e:ruin.asymp}
agrees with the standard univariate result on subexponential claims;
see e.g. Theorem 1.3.8 in
\cite{embrechts:kluppelberg:mikosch:1997}. If the claim arrival
process is Poisson, then this is \eqref{e:subexp.ruin} of Section \ref{sec:1dim}.
\end{remark}

\begin{proof} [Proof of Theorem \ref{ruintail}]
We start by observing that the function $H$ is proportional to the
tail of a subexponential distribution $F^I_A$ and, hence, can itself
be viewed as the tail of a subexponential distribution.
We can and will, for example,  simply refer to the ``long tail
property'' of $H$.

We use the ``one big jump" approach to heavy tailed large deviations;
see e.g. \cite{zachary:2004}, and the first step is to show that
\begin{equation} \label{e:Z.H}
\lim_{u\to\infty} \frac{\int_0^\infty P(\mathbf{Z} \in u A + v
  \mathbf{c}) \, dv}{H(u)} =1\,.
\end{equation}
Indeed, the upper bound in \eqref{e:Z.H} follows from the fact that
$A$ is increasing. For the lower bound, notice that, by Fatou's lemma,
it is enough to prove that that for each fixed $\by$,
$$
\lim_{u \to \infty}  \frac{\int_0^\infty  F(u A + v \mathbf{c} + y
  \mathbf{p}) \, dv}{H(u)} =1\,.
$$
This, however, follows from the fact
for sufficiently large $u > 0$, there exists some $u_1>0$ such that $(u
+ u_1)A + v \mathbf{c} \subset uA + v \mathbf{c} + y \mathbf{p}$, and
the long tail property of $H$.

We proceed to prove the lower bound in \eqref{e:ruin.asymp}.
Let $\mathbf{S}_n := \sum_{i = 1}^n \mathbf{Z}^{(i)}$,
$n=1,2,\ldots$. Let $\epsilon, \delta$ be small positive numbers, and
choose $K$ so large that
\begin{align*}
P\bigl(\mathbf{S}_n > -(K + n(1 + \epsilon)) \mathbf{c}\bigr) > 1 - \delta,
  \ n=1,2,\ldots\,.
\end{align*}
Define $M_n = \sup \{ u>0:\,  \mathbf{S}_i \in u A ~\text{for some
} 1 \leq i \leq n \}$ and  $M = \sup \{ u>0:\, \mathbf{S}_n
\in u A ~\text{for some }  n \}$. For $u > 0$,
\begin{align*}
\psi_{\mathbf{b}, L}(u) & = P(M > u) = \sum_{n \geq 0} P( M_n \leq
                          u,\, \mathbf{S}_{n+1} \in u A) \\
& \geq \sum_{n \geq 0} P\bigl( M_n \leq u,\, \mathbf{S}_n > -(K + n(1 +
  \epsilon)) \mathbf{c}, \, \mathbf{Z}^{(n+1)} \in u A + (K + n(1 +
  \epsilon)) \mathbf{c}\bigr) \\
 & \geq \sum_{n \geq 0} (1 - \delta - P(M_n > u)) P\bigl(
   \mathbf{Z}^{(n+1)} \in u A + K \mathbf{c} + n(1 + \epsilon)
   \mathbf{c}\bigr) \\
& \geq (1 - \delta - P(M > u)) \sum_{n \geq 0} P\bigl( \mathbf{Z} \in
  u A + K \mathbf{c} + n(1 + \epsilon) \mathbf{c}\bigr)\,.
\end{align*}
Rearranging, and using the monotonicity of $A$ and the long tail
property of $F_A^I$, we see that
\begin{align*}
\psi_{\mathbf{b}, L}(u) & \geq \frac{(1 - \delta) \sum_{n \geq 0} P(
                          \mathbf{Z} \in u A + K \mathbf{c} + n(1 +
                          \epsilon) \mathbf{c})}{1   + \sum_{n \geq 0}
                          P( \mathbf{Z} \in u A + K \mathbf{c} + n(1 +
                          \epsilon) \mathbf{c})} \\
& \sim  \frac{1 - \delta}{1+\epsilon}\int_0^\infty P\bigl( \mathbf{Z} \in
  u A + K \mathbf{c} + v \mathbf{c}\bigr) \, dv \\
& \sim  \frac{1 - \delta}{1+\epsilon}\int_0^\infty P\bigl( \mathbf{Z} \in
  u A +  v \mathbf{c}\bigr) \, dv,  \ u\to\infty\,.
\end{align*}
Letting  $\delta, \epsilon$ to $0$,  we have, thus, obtained the lower
bound in \eqref{e:ruin.asymp}. We proceed to prove a matching upper
bound.

Fix $0<\epsilon<1$. For $r > 0$, we define a sequence $(\tau_n)$ as
follows: we set $\tau_0 = 0$, and
\begin{align*}
\tau_1 = \inf \bigl\{ n \geq 1 :\,  \mathbf{S}_n \in r A - n (1 - \epsilon)
  \mathbf{c}\bigr\} \,.
\end{align*}
For $m \geq 2$, we set $\tau_m = \infty$
if $\tau_{m - 1} = \infty$. Otherwise, let
\begin{align*}
\tau_m = \tau_{m - 1} + \inf \bigl\{ n \geq 1 :\,  \mathbf{S}_{n +
  \tau_{m - 1}} - \mathbf{S}_{\tau_{m - 1}} \in r A - n (1 -
  \epsilon) \mathbf{c} \bigr\}.
\end{align*}
If we let $\gamma = P(\tau_1 < \infty)$, then for any $m \geq 1$,
$P(\tau_m < \infty) = \gamma^m$. By the positive safety loading
assumption, $\gamma \to 0$ as $r \to \infty$. Note that for  $u > 0$,
\begin{align*}
P\bigl(\tau_1<\infty, \ \mathbf{S}_{\tau_1} \in u A\bigr)
& = \sum_{n \geq 1} P(\tau_1 = n, \, \mathbf{S}_n  \in u A)
\leq  \sum_{n \geq 1}
P\bigl(\mathbf{S}_{n - 1} \in r A^c  - (n - 1) (1 -\epsilon) \mathbf{c},\,
\mathbf{S}_n \in u A\bigr)\,.
\end{align*}
By part (c) of Lemma \ref{basic}, $\mathbf{S}_n \in u A$ if and
only if
$\sup_{\mathbf{p} \in I_A} \mathbf{p}^T \mathbf{S}_n > u$. Further,
\begin{align*}
\sup_{\mathbf{p} \in I_A} \mathbf{p}^T \mathbf{S}_n
& \leq \sup_{\mathbf{p} \in I_A} \mathbf{p}^T \bigl(\mathbf{S}_{n - 1}
+ (n - 1) (1 - \epsilon) \mathbf{c}\bigr)
+ \sup_{\mathbf{p} \in I_A} \mathbf{p}^T \bigl(\mathbf{Z}^{(n)}
- (n - 1) (1 -  \epsilon) \mathbf{c}\bigr).
\end{align*}
Let $u>r$. If $\mathbf{S}_{n - 1} \in r A^c - (n - 1) (1 - \epsilon)
\mathbf{c}$, then $\sup_{\mathbf{p} \in I_A}
\mathbf{p}^T \bigl(\mathbf{\Gamma}_{n - 1} + (n - 1) (1 - \epsilon)
\mathbf{c}\bigr) \leq r$, so for $\sup_{\mathbf{p} \in I_A} \mathbf{p}^T
\mathbf{S}_n > u$ to hold, it must be the case that $\sup_{\mathbf{p}
  \in   I_A} \mathbf{p}^T \bigl(\mathbf{Z}^{(n)} - (n - 1) (1 -
\epsilon) \mathbf{c}\bigr) > u - r$, implying that $\mathbf{Z}^{(n)} \in
(u - r)   A + (n - 1) (1 - \epsilon) \mathbf{c}$.

Summing up, we see that, as $u\to\infty$,
\begin{align*}
P\bigl(\tau_1<\infty, \, \mathbf{S}_{\tau_1} \in u A\bigr) & \leq \sum_{n \geq 1}
  P\bigl(\mathbf{Z}^{(n)}   \in (u - r) A   + (n - 1) (1 -  \epsilon)  \mathbf{c}\bigr) \\
& \sim \int_0^\infty P(\mathbf{Z} \in (u - r) A + v (1 - \epsilon)
  \mathbf{c}) \, dv \\
& \sim\frac{1}{1-\epsilon}H(u - r)\,.
\end{align*}
Letting   $\epsilon \to 0$ and using the long tail property of $H$,
we obtain
\begin{equation} \label{e:ub.1}
\limsup_{u\to\infty} \frac{P\bigl(\tau_1<\infty, \,
  \mathbf{S}_{\tau_1} \in u A\bigr)}{H(u) }\leq 1\,.
\end{equation}

Let $(\BV^{(i)})$ be a sequence of independent identically distributed
random vectors whose law is the conditional law of $
\mathbf{S}_{\tau_1}$ given that $\tau_1 < \infty$. By \eqref{e:ub.1},
there is a distribution $B$ on $[0,\infty)$ such that $\overline{B}(u)
\sim  \gamma^{-1} H(u)$ as $u\to\infty$ and
$$
P\bigl(\BV^{(1)} \in uA\bigr) \leq \overline{B}(u) \ \ \text{for all $u\geq 0$.}
$$
Note, further, that by the definition of the sequence $(\tau_m)$, for
every $m\geq 0$, on the event $\{ \tau_m<\infty\}$, we have,
for $1 \leq i < \tau_{m+1}$, $\mathbf{S}_{\tau_m + i} -
\mathbf{S}_{\tau_m} \in r A^c - i (1 - \epsilon) \mathbf{c} \subset r
A^c - (1 - \epsilon) \mathbf{c}$. If $\mathbf{S}_{\tau_m} \in
(u - r) A^c + (1 - \epsilon) \mathbf{c}$, then  we have
$\mathbf{S}_{\tau_m + i} \in u A^c$. Hence, for the event
$\{\mathbf{S}_n \in u A$ for some $n\}$ to occur, we  must be have
$$
\mathbf{S}_{\tau_m} \in \bigl( (u - r) A + (1 - \epsilon)
\mathbf{c}\bigr) \cup uA \ \ \text{for some $m$.}
$$
Therefore, we can use Lemma \ref{upbound} to obtain
\begin{align*}
\psi_{\mathbf{b}, L}(u) & = P(M > u) \leq \sum_{m \geq 1} P\bigl(
\mathbf{S}_{\tau_m} \in\bigl(  (u - r) A +   (1 - \epsilon)
                          \mathbf{c}\bigr) \cup uA\bigr) \\
& \leq \sum_{m \geq 1} \gamma^m P\big(\BV^{(1)} + \dots + \BV^{(m)}
  \in (u - r) A\bigr) \\
& \leq \sum_{m \geq 1} \gamma^m \overline{B^{(m)}} (u - r).
\end{align*}
By the assumption, the $H$ is the tail of a subexponential
distribution, and, hence, $B$ is subexponential as well.
This implies that
\begin{align*}
\lim_{u \to \infty} \frac{\overline{B^{(m)}}(u )}{\overline{B}(u)} =   m\,,
\end{align*}
and that for any $\epsilon > 0$, there exists $K > 0$ such that for
all $u > 0$ and $m \geq 1$,
\begin{align*}
\frac{\overline{B^{(m)}}(u)}{\overline{B}(u)} \leq K (1 + \epsilon)^m.
\end{align*}
Since we can make $\gamma>0$ as small as we wish by choosing $r$
large, we can use the dominated convergence theorem to obtain
\begin{align*}
\limsup_{u \to \infty} \frac{\psi_{\mathbf{b}, L}(u)}{\gamma
  \overline{B} (u-r)} &  = \sum_{m \geq 1}
                        \gamma^{m - 1} m = \frac{1}{(1 - \gamma)^2}.
\end{align*}
Letting $r \to \infty$, which makes $\gamma \to 0$, we have that
\begin{align*}
\limsup_{u \to \infty} \frac{\psi_{\mathbf{b}, L}(u)}{ H (u)}  & \leq 1\,,
\end{align*}
which is the required upper bound in \eqref{e:ruin.asymp}.
\end{proof}

We finish this section by returning to the special case of
multivariate regularly varying claims. Recall that, by Proposition
\ref{mrvprop}, the distributions in $MRV(\alpha,\mu)$ are in
$\mathscr{S}_{\mathcal{R}}$. The asymptotic behaviour of the ruin
probability with the solvency set $L^c$ given by
(\ref{solvency}),  and multivariate regularly varying claims with
$\alpha>1$, was determined by \cite{hult:lindskog:2006a}. To state
their result, notice that the tail measure of a random vector $\BX$
(recall \eqref{e:regvar.d.st}) is determined up to a scaling by a
positive constant, and a different scaling in the tail measure can be
achieved by scaling appropriately the function $b$ in
\eqref{e:regvar.d.st}. Let us scale the tail measure $\mu$ in such a
way that it assigns unit mass to the complement of the unit ball in
$\bbr^d$. The norm we choose is unimportant, but for consistency with
the notation used elsewhere in the paper, let us use the $L_1$
norm. With this convention, we can restate \eqref{e:regvar.d.st} as
\begin{equation} \label{e:regvar.d.st0}
\frac{P(\BX\in u\cdot)}{P(\|\BX\|>u)} \vconv \mu
\end{equation}
vaguely  on $[-\infty,\infty]^d \setminus \{{\bf 0}\}$. It was shown by
\cite{hult:lindskog:2006a} that under the assumption
\eqref{e:regvar.d.st0} (and with the solvency set $L^c$ given by
(\ref{solvency})), the ruin probability satisfies
\begin{align} \label{mrvruin1}
\lim_{u \to \infty} \frac{\psi_{\mathbf{b}, L}(u)}{u P(\|\mathbf{X}\| >
  u)} = \int_0^\infty \mu(\mathbf{b} - L + v \mathbf{c}) \, dv.
\end{align}
We extend the above result to all ruin sets satisfying Assumption
\ref{ass:ruin.set}. To avoid a degenerate situation (and the resulting
complications in the notation) we will assume that $\mu\{ \bx:\,
x_i>0\}>0$ for each $i=1,\ldots, d$.

\begin{proposition} \label{mrvruin2}
Assume  that the ruin set $L$ satisfies Assumption
\ref{ass:ruin.set}. If the claim sizes satisfy \eqref{e:regvar.d.st0}
with $\alpha > 1$, then \eqref{mrvruin1} holds.
\end{proposition}

\begin{proof}
By Theorem \ref{ruintail}, it suffices to show that
\begin{align*}
\lim_{u \to \infty} \frac{\int_0^\infty P(\mathbf{X} \in u A + v
  \mathbf{c}) \, dv}{u P(|\mathbf{X}| > u)} = \int_0^\infty \mu(A + v
  \mathbf{c}) \, dv\,,
\end{align*}
which we proceed to do. By a change of variables,
\begin{align} \label{e:ch.var}
\frac{\int_0^\infty P(\mathbf{X} \in u A + v \mathbf{c}) \, dv}{u
  P(|\mathbf{X}| > u)} &  = \frac{\int_0^\infty P(\mathbf{X} \in u (A
                         + v \mathbf{c})) \, dv}{P(|\mathbf{X}| >  u)}\,,
\end{align}
and for every $v>0$,
\begin{align*}
\frac{P(\mathbf{X} \in u (A + v \mathbf{c}))}{P(|\mathbf{X}| > u)} \to \mu(A + v \mathbf{c})
\end{align*}
as $u \to \infty$. In the last step we use \eqref{e:regvar.d.st0}, and
the fact that the tail measure does not charge the boundary of sets in
$\calR$, shown in the proof of Proposition \ref{mrvprop}. Therefore,
we only need to justify taking the limit inside the integral in
\eqref{e:ch.var}. However, by the definition of the  set $A$,
\begin{align*}
\frac{P(\mathbf{X} \in u (A + v \mathbf{c}))}{P(|\mathbf{X}| > u)}
&\leq \sum_{i=1}^d \frac{P(X^{(i)}>ub_i+uvc_i)}{P(X^{(i)}>u)}\,.
\end{align*}
The non-degeneracy assumption on the measure $\mu$ implies that each
$X^{(i)}$ is itself regularly varying with exponent
$\alpha$. Therefore, by the Potter bounds, there are finite positive
constants $C_i, \, i=1,\ldots, d$, and a number $\vep\in (0,\alpha-1)$
such that for all $u\geq 1$,
$$
\frac{P(X^{(i)}>ub_i+uvc_i)}{P(X^{(i)}>u)} \leq
C_i(b_i+vc_i)^{-(\alpha-\vep)}, \ i=1,\ldots, d\,.
$$
Since the functions in the right hand side are integrable, the
dominated convergence theorem applies.
\end{proof}

%\nocite{*}

\bibliographystyle{mystyle}
\bibliography{bibfile}

%\bibliography{MultSub}{}

\end{document}